\documentclass[11pt,reqno]{amsart}
\usepackage{amscd,graphicx,psfrag,amsxtra}

\newcommand{\p}{\partial}
\newcommand{\Z}{\mathbb Z}

\newcommand{\C}{\mathbb C}
\newcommand{\M}{\mathcal M}
\newcommand{\R}{\mathbb R}

\renewcommand{\phi}{\varphi}

\newcommand{\SW}{\operatorname{SW}\,}

\newcommand{\ind}{\operatorname{ind}}

\newcommand{\spin}{\,\operatorname{spin}}

\newcommand{\sign}{\operatorname{sign}}

\newcommand{\Dir}{\,\operatorname{Dir}}
\newcommand{\Sign}{\,\operatorname{Sign}}
\renewcommand{\(}{((}
\renewcommand{\)}{))}
\newcommand{\lt}{\left(\hspace{-0.06in}\left(}
\newcommand{\rt}{\right)\hspace{-0.06in}\right)}
\newcommand{\mmod}{\hspace{-0.07in}\mod}

\newtheorem{theorem}{Theorem}[section]

\newtheorem{lemma}[theorem]{Lemma}

\theoremstyle{definition}

\title{The {\Large $\bar\mu$}--invariant of Seifert fibered homology spheres 
and the Dirac operator}
\thanks{The first author was partially supported by NSF Grant 0804760. 
The second author was partially supported by the Max-Planck-Institut 
f\"ur Mathematik in Bonn, Germany}
\author[Daniel Ruberman]{Daniel Ruberman}
\address{Department of Mathematics, MS 050\newline\indent Brandeis
University \newline\indent Waltham, MA 02454}
\email{\rm{ruberman@brandeis.edu}}
\author[Nikolai Saveliev]{Nikolai Saveliev}
\address{Department of Mathematics\newline\indent
University of Miami \newline\indent PO Box 249085
\newline\indent Coral Gables, FL 33124}
\email{\rm{saveliev@math.miami.edu}}

\begin{document}
\begin{abstract}
We derive a formula for the $\bar\mu$--invariant of a Seifert fibered 
homology sphere in terms of the $\eta$--invariant of its Dirac operator.  
As a consequence, we obtain a vanishing result for the index of 
certain Dirac operators on plumbed 4-manifolds bounding such spheres.
\end{abstract}

\maketitle
\section{Introduction}\label{S:intro}
The $\bar\mu$--invariant is an integral lift of the Rohlin invariant for 
plumbed homology 3-spheres defined by Neumann \cite{N} and Siebenmann 
\cite{Sb}. It has played an important role in the study of homology 
cobordisms of such homology spheres. Fukumoto and Furuta \cite{FuFu} and 
Saveliev \cite{Sav2} showed that the $\bar\mu$--invariant is an obstruction 
for a Seifert fibered homology sphere to have finite order in the integral 
homology cobordism group $\Theta^3_H$; this fact allowed them to make 
progress on the question of the splittability of the Rohlin homomorphism 
$\rho: \Theta^3_H \to \Z_2$. Ue \cite{Ue2} and Stipsicz \cite{St} studied 
the behavior of $\bar\mu$ with respect to rational homology cobordisms. 

In the process, the $\bar\mu$--invariant has been interpreted in several
different ways\,: as an equivariant Casson invariant in \cite{CS}, as a 
Lefschetz number in instanton Floer homology in \cite{RS} and \cite{Sav1},
and as the correction term in Heegaard Floer theory in \cite{St} and 
\cite{Ue1}.  

More recently, $\bar\mu$ appeared in our paper \cite{MRS} in connection 
with a Seiberg-Witten invariant $\lambda_{\,\SW}$ of a homology $S^1\times 
S^3$. We conjectured that
\begin{equation}\label{E:conj}
\lambda_{\,SW} (X) = - \bar\mu (Y)
\end{equation}
for any Seifert fibered homology sphere $Y = \Sigma(a_1,\ldots,a_n)$ and
the mapping torus $X$ of a natural involution on $Y$ viewed as 
a link of a complex surface singularity. This conjecture will be explained 
in detail and proved in Section \ref{S:conj}. For the purposes of this 
introduction, we will only mention that its proof will rely on the 
following identity.

\begin{theorem}\label{T:main}
Let $Y = \Sigma (a_1,\ldots,a_n)$ be a Seifert fibered homology sphere 
oriented as the link of complex surface singularity and endowed with a 
natural metric realizing the Thurston geometry on $Y$; see \cite{Scott}. 
Then
\begin{equation}\label{E:main}
\frac 1 2\;\eta_{\Dir} (Y) + \frac 1 8\;\eta_{\Sign} (Y) = -\bar\mu (Y),
\end{equation}
where $\eta_{\Dir} (Y)$ and $\eta_{\Sign}(Y)$ are the $\eta$--invariants 
of, respectively, the Dirac operator and the odd signature operator on 
$Y$.
\end{theorem}

In addition, identity \eqref{E:main} will be used to extend the vanishing 
result of Kronheimer for the index of the chiral Dirac operator on the 
$E_8$ manifold bounding $\Sigma (2,3,5)$; see \cite[Lemma 2.2]{Kron} and 
\cite[Proposition 8]{Fro}. Let $Y$ be a Seifert fibered homology sphere 
as above, and $X$ a plumbed manifold with boundary $Y$ and a Riemannian 
metric which is a product near the boundary. Associated with $X$ is the 
integral Wu class $w \in H_2(X;\Z)$ which will be described in detail in 
Section \ref{S:dirac}.  

\begin{theorem}\label{T:dirac}
Let $D^+_L(X)$ be the $\spin^c$ Dirac operator on $X$ with $c_1(L)$ dual to 
the class $w \in H_2 (X;\Z)$,  and with the Atiyah--Patodi--Singer boundary 
condition. Then $\ind D^+_L (X) = 0$. In particular, if $X$ is spin then 
$w$ vanishes and $\ind D^+ (X) = 0$.
\end{theorem}

The next four sections of the paper will be devoted to the proof of Theorem 
\ref{T:main}. We will proceed by expressing both sides of (\ref{E:main}) in 
terms of Dedekind--Rademacher sums and by comparing the latter expressions 
using the reciprocity law and some elementary calculations. Theorem 
\ref{T:dirac} will be proved in Section \ref{S:dirac}, and Conjecture 
\eqref{E:conj} in Section \ref{S:conj}.  Our notation and conventions for 
Seifert fibered homology spheres will follow~\cite{saveliev:spheres}.


\section{The $\eta$--invariants}
Let $p > 0$ and $q > 0$ be pairwise relatively prime integers, and $x$ and
$y$ arbitrary real numbers.  The Dedekind--Rademacher sums were defined in 
\cite{Rad} by the formula 
\[
s(q,p;x,y) = \sum_{\mu \mmod p} \lt \frac {\mu + y} p \rt
 \lt \frac {q(\mu + y)} p + x \rt
\]
where, for any real number $r$, we set $\{r\} = r - [r]$ and 
\[
\( r \) = 
\begin{cases}
\; 0, &\quad\text{if\; $r \in \Z$}, \\
\, \{r\} - 1/2, &\quad\text{if\; $r \notin \Z$}.
\end{cases}
\]
It is clear that $s(q,p;x,y)$ only depends on $x$, $y\mmod 1$. When both $x$
and $y$ are integers, we get back the usual Dedekind sums

\begin{equation}\label{E:ded}
s(q,p) = \sum_{\mu \mmod p} \lt \frac {\mu} p \rt \lt \frac {q \mu} p \rt.
\end{equation}

\bigskip

The left-hand side of (\ref{E:main}) was expressed by Nicolaescu \cite{Nic2} 
in terms of Dedekind--Rademacher sums. Note that since the $a_i$ are coprime, 
at most one of them is even; if that occurs then we will choose the even one 
to be $a_1$.\\[2ex]
\textbf{Odd case:} if all $a_1,\ldots,a_n$ are odd then, according to the
formula (1.9) of \cite{Nic2}, we have
\begin{multline}\label{E:eta-odd}
\frac 1 2\;\eta_{\Dir} (Y) + \frac 1 8\;\eta_{\Sign} (Y) = 
- \frac 1 {8\,a_1\cdots a_n} + \\ + \frac 1 8 + 
\frac 1 2\,\sum_{i=1}^n s(a_1\cdots a_n/a_i, a_i) + 
\sum_{i=1}^n s(a_1\ldots a_n/a_i,a_i; 1/2,1/2).
\end{multline}

\bigskip

\noindent
\textbf{Even case:} if $a_1$ is even then, according to the formula (1.6) 
of \cite{Nic2},
\begin{multline}\label{E:eta-even}
\frac 1 2\;\eta_{\Dir} (Y) + \frac 1 8\;\eta_{\Sign} (Y) = \\ = \frac 1 8 + 
\frac 1 2\,\sum_{i=1}^n s(a_1\cdots a_n/a_i, a_i) + 
\sum_{i=1}^n s(a_1\ldots a_n/a_i,a_i; 1/2,1/2).
\end{multline}


\section{The $\bar\mu$--invariant}
Let $\Sigma$ be a plumbed integral homology sphere, and let $X$ be an 
oriented plumbed 4-manifold such that $\p X = \Sigma$. The integral Wu 
class $w \in H_2 (X;\Z)$ is the unique homology class which is 
characteristic and whose coordinates are either 0 or 1 in the natural 
basis in $H_2 (X;\Z)$ represented by embedded 2-spheres. According to 
Neumann \cite{N}, the integer
\[
\bar\mu (\Sigma) = \frac 1 8\;(\,\sign (X) - w\cdot w)
\]
is independent of the choices in its definition and reduces modulo 2 to 
the Rohlin invariant of $\Sigma$. It is referred to as the 
$\bar\mu$--invariant. 

Let us now restrict ourselves to the case of $Y = \Sigma(a_1,\ldots,a_n)$. 
Choose integers $b_1,\ldots,b_n$ so that 
\begin{equation}\label{E:one}
a_1\cdots a_n\;\sum_{i=1}^n\;\frac {b_i}{a_i}\; 
=\;\sum_{i=1}^n\;b_i\,a_1\cdots a_n/a_i\; = 1.
\end{equation}
Note that each $b_i$ is defined uniquely modulo $a_i$. Then we have the 
following formulas for the $\bar\mu$--invariant; see \cite[Corollary 2.3]{N} 
and \cite[Theorem 6.2]{NR}.

\medskip

\noindent\textbf{Odd case:}  if all $a_1,\ldots,a_n$ are odd then
\[
-\bar\mu (Y) = \frac 1 8 - \frac 1 8\; \sum_{i=1}^n\,(c(a_i,b_i) + \sign b_i).
\]

\medskip

\noindent
\textbf{Even case:}  if $a_1$ is even, choose $b_i$ so that all $a_i - b_i$ 
are all odd (by replacing, if necessary, $b_i$ by $b_i \pm a_i$ for each $i 
> 1$, and then adjusting $b_1$ accordingly). Then  
\[
-\bar\mu (Y) = \frac 1 8 - \frac 1 8\; \sum_{i=1}^n\,c(a_i - b_i,a_i).
\]

\medskip

Here, the integers $c(q,p)$ are defined for coprime integer pairs $(q,p)$
with $q$ odd as follows. First, assume that both $p$ and $q$ are positive.
Then

\[
c(q,p)\; =\; - \frac 1 p \sum_{\xi^p = -1} \frac {(\xi + 1)(\xi^q + 1)}
{(\xi - 1)(\xi^q - 1)}\; =\; 
\frac 1 p \sum_{\substack{k = 1 \\ k\;\text{odd}}}^{2p - 1} 
\cot\left(\frac{\pi k}{2p}\right)\cot\left(\frac{\pi qk}{2p}\right)
\]

\medskip
The integers $c(q,p)$ show up in the book~\cite[Theorem 1, pp.~102--103]{HZ} 
under the name $-t_p(1,q)$. We can use that theorem together with formula 
(6) on page 100 of \cite{HZ} to write
\begin{equation}\label{E:cpq}
c (q,p) = - 4s(q,p) + 8s(q,2p).
\end{equation}

Next,  the above definition of $c(q,p)$ is extended to both positive and 
negative $p$ and $q$ by the formula $c(q,p) = \sign(pq)\,c(|q|,|p|)$.  Using 
(\ref{E:cpq}), we can write the above formulas for the $\bar\mu$--invariant 
in the following form.

\medskip

\noindent\textbf{Odd case:}  if all $a_1,\ldots,a_n$ are odd then
\begin{multline}\label{E:mu-odd}
-\bar\mu (Y) = \frac 1 8 - \frac 1 8 \sum_{i=1}^n\, \sign b_i + \frac 1 2\,
\sum_{i=1}^n \sign b_i\cdot s(a_i,|b_i|) \\ 
- \sum_{i=1}^n\, \sign b_i\cdot s(a_i,2|b_i|).
\end{multline}

\medskip

\noindent
\textbf{Even case:} if $a_1$ is even and $b_i$ are chosen so that $a_i - b_i$ 
are all odd, then  
\begin{equation}\label{E:mu-even}
-\bar\mu (Y) = \frac 1 8 + \frac 1 2\,\sum_{i=1}^n\, s(a_i - b_i,a_i)
- \sum_{i=1}^n s(a_i - b_i,2a_i)
\end{equation}

\medskip\noindent
In the latter formula, we used a natural extension of the Dedekind sum 
$s(q,p)$ to the negative values of $q$ as an odd function in $q$; it is 
still given by the formula (\ref{E:ded}). We will continue to assume,
however, that $p$ in $s(q,p)$ is positive.


\section{The odd case}
In this section, we will show that the right hand sides of (\ref{E:eta-odd}) 
and (\ref{E:mu-odd}) are equal to each other, thus proving the formula 
(\ref{E:main}) in the case when all $a_1,\ldots,a_n$ are odd. 

\begin{lemma}\label{L:one}
For any integers $a > 0$ and $b$, $c$ such that $bc = 1\mmod a$ we have  
$s(c,a) = s(b,a)$.
\end{lemma}

\begin{proof} Observe that $bc = 1\mmod a$ implies that $b$ and $a$ are
coprime hence
\begin{multline}\notag
s(c,a) = \sum_{\mu\mmod a} 
\lt\frac{\mu}{a}\rt \lt\frac{\mu c}{a}\rt 
= \sum_{\mu\mmod a} 
\lt\frac{\mu b}{a}\rt \lt\frac{\mu b c}{a}\rt \\
= \sum_{\mu\mmod a} 
\lt\frac{\mu b}{a}\rt \lt\frac{\mu}{a}\rt
=  s(b,a).
\end{multline}
\end{proof}

\begin{lemma}\label{L:two}
For any coprime positive integers $a$ and $b$ such that $a$ is odd, 
\[
\frac 1 2\,s(a,b) - s(a,2b)\; =\; - s(a,b;0,1/2) - \frac 1 2\,s(a,b).
\]
\end{lemma}

\begin{proof}
The proof goes by splitting the summation over $\mu\mmod 2b$ in $s(a,2b)$ 
into two summations, one over even $\mu = 2\nu$, and the other over odd 
$\mu = 2\nu + 1$. More precisely,
\begin{alignat*}{1}
s(a,2b) 
& = \sum_{\mu\mmod 2b} \lt\frac{\mu}{2b}\rt \lt\frac{a\mu}{2b}\rt \\
& = \sum_{\nu\mmod b} \lt\frac{\nu}{b}\rt \lt\frac{a\nu}{b}\rt + 
\sum_{\nu\mmod b} \lt\frac{2\nu+1}{2b}\rt \lt\frac{a(2\nu+1)}{2b}\rt \\
& = \sum_{\nu\mmod b} \lt\frac{\nu}{b}\rt \lt\frac{a\nu}{b}\rt + 
\sum_{\nu\mmod b} \lt\frac{\nu+1/2}{b}\rt \lt\frac{a(\nu+1/2)}{b}\rt \\
& = s(a,b) + s(a,b;0,1/2).
\end{alignat*}
The statement of the lemma now follows. 
\end{proof}

\medskip

Applying Lemma \ref{L:one} with $a = a_i$, $b = b_i$ and $c = a_1\cdots 
a_n/a_i$, and Lemma \ref{L:two} with $a = a_i$ and $b =
|b_i|$ respectively to the formulas (\ref{E:eta-odd}) and (\ref{E:mu-odd}),
we see that all we need to do is verify the following identity
\begin{multline}\label{E:int1}
- \frac 1 {8\,a_1\cdots a_n} + \frac 1 2\; \sum_{i=1}^n s(b_i,a_i) + 
\sum_{i=1}^n s(a_1\cdots a_n/a_i,a_i; 1/2,1/2) = \\
- \sum_{i=1}^n\;\sign b_i \cdot \left( \frac 1 8 + s(a_i,|b_i|;0,1/2)
+ \frac 1 2 s(a_i,|b_i|)\right).
\end{multline}

\medskip\noindent
Use the reciprocity laws (see for instance Appendix in \cite{Nic2}) to
obtain

\[
s(a_i,|b_i|;0,1/2) = -s (|b_i|,a_i; 1/2,0) + \frac {2b_i^2 - a_i^2 - 1}
{24\,a_i |b_i|}
\]
and 
\[
s(a_i,|b_i|) = - s(|b_i|,a_i) - \frac 1 4 + \frac {a_i^2 + b_i^2 + 1}
{12\,a_i |b_i|}.
\]

\medskip\noindent
Substituting the latter two formulas into (\ref{E:int1}) and keeping in mind 
that
\[
\sum_{i=1}^n\; \frac {b_i}{a_i}\; =\; \frac 1 {a_1\cdots a_n}
\]
because of (\ref{E:one}), we reduce verification of (\ref{E:int1}) to 
proving the following lemma (we write $\sign b_i\cdot s(|b_i|,a_i;1/2,0)
= s(b_i,a_i;1/2,0)$).

\begin{lemma}\quad
$s(b_i,a_i; 1/2,0) = s(a_1\cdots a_n/a_i,a_i; 1/2,1/2)$.
\end{lemma}

\begin{proof}
One can easily see that the identity that needs to be verified, 
\begin{multline}\notag
\sum_{\mu\mmod a_i} \lt \frac {\mu}{a_i}\rt \lt \frac {b_i\mu}{a_i} + 
\frac 1 2 \rt = \\ 
\sum_{\nu\mmod a_i} \lt \frac {\nu + 1/2}{a_i}\rt 
\lt \frac {(\nu + 1/2)a_1\cdots a_n/a_i}{a_i} + \frac 1 2 \rt,
\end{multline}
follows by substitution $\nu = b_i\,\mu + (a_i - 1)/2 \mmod a_i$.
\end{proof}


\section{The even case}
In this section, we will prove the equality of the right hand sides of 
the formulas (\ref{E:eta-even}) and (\ref{E:mu-even}) and hence prove 
(\ref{E:main}) in the even case.

\begin{lemma}\label{L:three}
\quad $s(a_i - b_i,a_i) = - s(b_i,a_i)$.
\end{lemma}

\begin{proof}
Since $\( x \)$ is an odd function in $x$, we have 
\begin{multline}\notag
s(a_i - b_i,a_i) 
= \sum_{\mu\mmod a_i} \lt\frac {\mu}{a_i}\rt \lt \frac {\mu (a_i - b_i)}
{a_i} \rt \\
= \sum_{\mu\mmod a_i} \lt\frac {\mu}{a_i}\rt \lt \frac {-\mu b_i}{a_i}\rt
= -s(b_i,a_i).
\end{multline}
\end{proof}

Using Lemma \ref{L:one} and Lemma \ref{L:three}, we reduce our task to 
showing that, for every $i = 1,\ldots,n$, 
\begin{equation}\label{E:four}
s(a_1\cdots a_n/a_i,a_i;1/2,1/2) + s(a_i - b_i,2a_i) + s(b_i,a_i) = 0.
\end{equation}

\begin{lemma}
For any coprime integers $a > 0$ and $c > 0$, we have 
\[
s(c,a;1/2,1/2) + s(a - c,2a) + s(c,a) = 0.
\]
\end{lemma}

\begin{proof}
Like in the proof of Lemma \ref{L:two}, we will break the summation over 
$\mu\mmod 2a$ in $s(a - c,2a)$ into two summations, one over $\mu = 2\nu$ 
and the other over $\mu = 2\nu + 1$. More precisely,
\begin{multline}\notag
s(a - c,2a) 
= \sum_{\mu\mmod 2a} \lt\frac{\mu}{2a}\rt \lt\frac{(a - c)\mu}{2a}\rt \\
= -\sum_{\nu\mmod a} \lt\frac{\nu}{a}\rt \lt\frac{c\nu}{a}\rt -
\sum_{\nu\mmod a} \lt\frac{2\nu+1}{2a}\rt \lt\frac{c(2\nu+1)}{2a} + 
\frac 1 2 \rt \\
= -s(c,a) - s(c,a;1/2,1/2).
\end{multline}
\end{proof}

We will apply the above lemma with $a = a_i$ and $c = a_1\cdots a_n/a_i$ to
obtain $s(a_1\cdots a_n/a_i,a_i;1/2,1/2) + s(a_i - a_1\cdots a_n/a_i,2a_i) +
s(a_1\cdots a_n/a_i,a_i) = 0$. Using Lemma \ref{L:one} to replace 
$s(a_1\cdots a_n/a_i,a_i)$ in the above formula by $s(b_i,a_i)$, we see that
the proof of (\ref{E:four}) will be complete after we prove the following 
formula.

\begin{lemma}
\quad $s(a_i - a_1\cdots a_n/a_i,2a_i) = s(a_i - b_i,2a_i)$. 
\end{lemma}

\begin{proof}
This is immediate from Lemma \ref{L:one} once we show 
that $(a_i - b_i)(a_i - a_1\cdots a_n/a_i) = 1\mmod 2a_i$. We will
consider two separate cases. If $i = 1$ then $a_1$ is even and $b_1$ is 
odd. Multiply out to obtain $(a_1 - b_1)(a_1 - a_2\cdots a_n) = a_1^2 +
b_1 a_2 \cdots a_n - a_1 (b_1 + a_2 \cdots a_n)$. Obviously, the first 
and the last summands are equal to zero modulo $2a_1$ because $a_1$ and
$(b_1 + a_2\cdots a_n)$ are even. Use the formula (\ref{E:one}) to write
$b_1 a_2 \cdots a_n = 1 - a_1 (b_2 a_2\cdots a_n + \ldots + b_n a_2
\cdots a_{n-1})$ and observe that the $b_2,\ldots, b_n$ are all even. 
This completes the proof in the case of $i = 1$.

Now suppose that $i \ge 2$. Since $a_i$ and 2 are coprime, it is enough 
to check separately that $(a_i - b_i)(a_i - a_1\cdots a_n/a_i)$ is 1 mod 
$a_i$ and 1 mod 2. The former is clear from (\ref{E:one}), and the latter 
follows from the observation that both $a_i - b_i$ and $a_i - a_1\cdots 
a_n/a_i$ are odd. 
\end{proof}


\section{Proof of Theorem \ref{T:dirac}}\label{S:dirac}
Endow $Y = \Sigma(a_1,\ldots,a_n)$ with a natural metric realizing the 
Thurston geometry on $Y$; see \cite{Scott}. Let $X$ be a plumbed manifold 
with boundary $\p X = Y$ and with metric that restricts to the metric on 
$Y$ and is a product near the boundary. If $X$ is spin, the 
Atiyah--Patodi--Singer index theorem~\cite{aps:I} asserts that
\begin{equation}\label{E:aps}
\frac 1 2\;\eta_{\Dir} (Y) + \frac 1 8\;\eta_{\Sign} (Y)\; =\; -\ind D^+ (X) 
- \frac 1 8\,\sign (X).
\end{equation}

\smallskip\noindent
Here, we used the fact that the Dirac operator on $Y$ has zero kernel; see 
Nicolaescu~\cite[Section 2.3]{Nic3}. On the other hand, it follows from 
the definition of the $\bar\mu$--invariant that $w = 0$ and hence
\[
\bar\mu (Y)\; =\; \frac 1 8\,\sign (X).
\]

\smallskip\noindent
The identity (\ref{E:main}) then implies that $\ind D^+ (X) = 0$. The 
special case of this when $Y$ is the Poincar\'e homology sphere 
$\Sigma(2,3,5)$ and $X$ is the negative definite $E_8$ manifold was proved 
by Kronheimer~\cite{Kron}.

If $X$ is not spin, for any choice of $\spin^c$ structure on $X$ with 
determinant bundle $L$ we have  
\[
\frac 1 2\;\eta_{\Dir} (Y) + \frac 1 8\;\eta_{\Sign} (Y)\; =\; 
-\ind D^+_L (X) - \frac 1 8\,\sign (X) + \frac 1 8\,c_1 (L)^2.
\]

\smallskip\noindent
(Compare with formula (1.37) in \cite{Nic1}). If the $\spin^c$ structure is 
such that $c_1 (L)$ is dual to $w \in H_2 (X;\Z)$ then
\[
\bar\mu (Y) = \frac 1 8\,(\sign (X) - w\cdot w) = \frac 1 8\,(\sign (X)
- c_1 (L)^2),
\]
and (\ref{E:main}) again implies that $\ind D^+_L (X) = 0$. This completes
the proof of Theorem \ref{T:dirac}.


\section{The invariant $\lambda_{\,\SW}$}\label{S:conj}
Let $X$ be a homology $S^1\times S^3$, by which we mean a closed oriented 
spin smooth 4-manifold with the integral homology of $S^1\times S^3$. For 
a generic pair $(g,\beta)$ consisting of a metric $g$ on $X$ and a 
perturbation $\beta \in \Omega^1(X,i\R)$, the Seiberg--Witten moduli 
space $\M(X,g,\beta)$ has finitely many irreducible points. It is 
oriented by a choice of homology orientation, that is, a generator $1 \in 
H^1 (X;\Z)$. Let $\#\,\M(X,g,\beta)$ denote the signed count of the points 
in this space. To counter the dependence of $\#\,\M(X,g,\beta)$ on the 
choice of $(g,\beta)$, we introduced in \cite{MRS} a correction term, 
$w(X,g,\beta)$, and proved that the quantity
\[
\lambda_{\,\SW} (X)\; =\; \#\,\M (X,g,\beta) - w (X,g,\beta)
\]

\smallskip\noindent
is an invariant of $X$ which reduces modulo 2 to its Rohlin invariant. 
The precise definition of the correction term is as follows. 

Let $Y \subset X$ be a smooth connected 3-manifold dual to the generator 
$1 \in H^1(X;\Z)$ and choose a smooth compact spin manifold $Z$ with 
boundary $Y$. Cutting $X$ open along $Y$ we obtain a cobordism $W$ from 
$Y$ to itself, which we use to construct the periodic-end manifold
\[
Z_+ = Z\,\cup\,W\,\cup\,W\ldots\cup\,W\,\cup\ldots
\]
The metric $g$ and perturbation $\beta$ extend to an end-periodic metric 
and, respectively, perturbation, on $Z_+$. This leads to the end-periodic 
perturbed Dirac operator $D^+ (Z_+) + \beta$, where $\beta$ acts via 
Clifford multiplication. We prove that $D^+(Z_+) + \beta$ is Fredholm 
in the usual Sobolev $L^2$-completion for generic $(g,\beta)$. The 
correction term is then defined as
\[
w (X,g,\beta)\; =\; \ind_{\C}\,(D^+ (Z_+) + \beta) + \frac 1 8\,\sign\,(Z).
\]

\smallskip
View $Y = \Sigma(a_1,\ldots,a_n)$ as a link of a complex surface singularity 
and let $X$ be the mapping torus of the involution on $Y$ induced by complex 
conjugation. The metric $g$ realizing the Thurston geometry on $Y$ is 
preserved by this involution and hence gives rise to a natural metric on $X$
called again $g$. We showed in \cite[Section 10]{MRS} that the pair $(g,0)$ 
is generic and that the space $\M(X,g,0)$ is empty. One can easily see that 
the manifold $Z_+$ has a product end and hence the correction term can be 
computed as in \eqref{E:aps} using the Atiyah--Patodi--Singer index theorem\,:
\[
w(X,g,0)\; =\; -\frac 1 2\;\eta_{\Dir} (Y) - \frac 1 8\;\eta_{\Sign} (Y).
\]

\smallskip\noindent
The conjecture \eqref{E:conj} now follows from Theorem \ref{T:main}. 


\end{document}